\documentclass[10pt]{amsart}
\usepackage{amssymb,latexsym,amsmath,amsfonts}
\usepackage[mathscr]{eucal}

\numberwithin{equation}{section}

\theoremstyle{definition}
\newtheorem{definition}{Definition}[section]

\theoremstyle{remark}
\newtheorem{remark}[definition]{Remark}

\theoremstyle{plain}
\newtheorem{theorem}[definition]{Theorem}
\newtheorem{result}[definition]{Result}

\newtheorem{proposition}[definition]{Proposition}

\newcommand{\zt}{\zeta}
\newcommand{\zbar}{\overline{z}}

\newcommand{\zahl}{\mathbb{Z}}  

\newcommand{\nat}{\mathbb{N}}

\newcommand{\rem}{\mathcal{R}}
\newcommand{\Lam}{\Lambda}

\newcommand\partl[2]{\frac{\partial{#1}}{\partial{#2}}}

\newcommand{\Dsc}{\overline{D}}

\newcommand{\smoo}{\mathcal{C}}
\newcommand{\hol}{\mathcal{O}}

\newcommand{\er}{{\sf Re}}
\newcommand{\mi}{{\sf Im}}

\newcommand{\bcdot}{\boldsymbol{\cdot}}
\newcommand\cis[1]{e^{i{#1}}}
\newcommand\ncis[1]{e^{-i{#1}}}

\newcommand{\lrarw}{\longrightarrow}

\newcommand{\doot}{\boldsymbol{.}}
\newcommand{\gee}{\mathcal{G}}
\newcommand{\hft}{\mathcal{H}}

\newcommand{\CC}{\mathbb{C}^2}

\newcommand{\cplx}{\mathbb{C}} 
\newcommand{\RR}{\mathbb{R}^2}
\newcommand{\rea}{\mathbb{R}}

\newcommand{\srf}{\mathcal{S}}   

\newcommand\gra[1]{\varGamma_{{#1}}}

\begin{document}

\title[Local Polynomial convexity]{Polynomial approximation, local polynomial \\
convexity, and degenerate CR singularities -- II}
\author{Gautam Bharali}

\thanks{This work is supported by the DST via the Fast Track grant SR/FTP/MS-12/2007 and
by the UGC under DSA-SAP, Phase~IV}

\address{Department of Mathematics, Indian Institute of Science, Bangalore 560012, India}
\email{bharali@math.iisc.ernet.in}


\begin{abstract} 
We provide some conditions for the graph of a H{\"o}lder-continuous function on $\Dsc$, 
where $\Dsc$ is a closed disc in $\cplx$, to be polynomially convex. Almost all
sufficient conditions known to date --- provided the function (say $F$) is smooth --- arise from
versions of the Weierstrass Approximation Theorem on $\Dsc$. These conditions often fail to yield
any conclusion if ${\rm rank}_\rea DF$ is not maximal on a sufficiently large subset of $\Dsc$. 
We bypass this difficulty by introducing a technique that relies on the interplay of certain
plurisubharmonic functions. This technique also allows us to make some observations on the 
polynomial hull of a graph in $\CC$ at an isolated complex tangency.
\end{abstract}
\maketitle

\section{Introduction and statement of results}\label{S:intro}

This paper has evolved from the following two considerations:
\begin{itemize}
\item Let $\Dsc$ be a closed disc in $\cplx$ and let $F\in\smoo(\Dsc)$. There are
numerous results that provide sufficient conditions for the uniform algebra on $\Dsc$
generated by $z$ and $F$ to equal $\smoo(\Dsc)$; 
see \cite{wermer:ad64, preskenis:apzaf78,
o'farrellPreskenis:ap2d84, duval:edpc88, o'farrellPreskenis:uap2f89,
bharali:palpcdCRs06}. These conditions are sufficient, naturally, for the graph of 
$F$ to be polynomially convex. The aforementioned results require --- either
by explicit fiat or through some {\em a priori} condition on $F$ --- that $F^{-1}\{F(\zt)\}$ 
be at most countable for a.e. $\zt\in\Dsc$. This is troublesome because it excludes, for 
instance, $\cplx(\cong\RR)$-valued functions in $\smoo^1(\Dsc)$ having ${\rm rank}_{\rea}DF<2$ 
on a non-empty open subset of $\Dsc$. It would thus be useful to devise techniques that allow 
us to detect polynomial convexity  without imposing such restrictions.

\item In a recent work, Dieu~\&~Chi \cite{dieuChi:lpccgCC09} employed a technique 
that can be applied to situations very different from the one that they study. Their idea,
suitably adapted to the given context, might serve as quite a general tool in the study 
of polynomial convexity of graphs. We wish to argue this case by presenting a
couple of adaptations of their idea.
\end{itemize} 

\noindent{Given $F$ as above and a set $S\subseteq{\sf Dom}(F)$, we shall write
$\gra{S}(F):={\sf Graph}(F)\cap(S\times\cplx)$. Let us understand some of the known sufficient 
conditions on $F$, for $\gra{\Dsc}(F)$ to be polynomially convex, by examining critically a 
representative result selected from the aforementioned papers. Hence consider:}

\begin{result}[\cite{bharali:palpcdCRs06}, Theorem~1.1]\label{R:bharali}
Let $F$ be a complex-valued continuous function on a closed disc
$\Dsc\Subset\cplx$. Suppose that there is a set $E\subset\Dsc$ having zero Lebesgue measure such
that 
\begin{itemize}
\item $F^{-1}\{F(\zt)\}$ is at most countable $\forall\zt\in\Dsc\setminus E$;
\item For each $\zt\in\Dsc\setminus E$, there exists an open sector $S_\zt\varsubsetneq\cplx$
with vertex at $0$ such that
\[
(z-\zt)(F(z)-F(\zt)) \ \in \ S_\zt \quad\forall z\in\Dsc\setminus F^{-1}\{F(\zt)\}.
\]
\end{itemize}
Then, $[z,F]_{\Dsc}=\smoo(\Dsc)$. In particular, $\gra{\Dsc}(F)$ is polynomially convex.  
\end{result} 
\smallskip

\noindent{To clarify before we proceed further: $[z,F]_{\Dsc}$ denotes the uniform algebra
on $\Dsc$ generated by $z$ and $F$.}
\smallskip

Result~\ref{R:bharali} is fairly representative of the key results in many of the papers 
cited above. It is proved using ideas very similar to those in \cite{wermer:ad64, preskenis:apzaf78, 
o'farrellPreskenis:ap2d84}. Observe:
\begin{itemize}
\item[I)] the requirement that $F^{-1}\{F(\zt)\}$ be countable for a.e. $\zt\in\Dsc$, and recall
the difficulties that it represents. Some version of this condition appears in 
\cite{wermer:ad64, preskenis:apzaf78, o'farrellPreskenis:ap2d84}, and is essential to ensuring that
the Cauchy transform of any measure $\mu\in\smoo(\Dsc)^\star$ that annihilates $[z,F]_{\Dsc}$ 
vanishes a.e. The approximation result then follows from an argument in \cite[Theorem~4]{bishop:mbfa59}.
This strategy was first used in \cite{wermer:ad64}.

\item[II)] that the truth value of the hypothesis of Result~\ref{R:bharali} is, in general, 
{\em not} preserved when $F$ is replaced by $\widetilde{{}^\psi F}$, 
where $\widetilde{{}^\psi F}$ is that function of which 
$\psi(\gra{\Dsc}(F))$ is the graph, for a given $\psi\in{\rm Aut}(\CC)$ that preserves 
$\Dsc\times\cplx$. In contrast, polynomial convexity (or the lack thereof) {\em is} preserved
under such a $\psi$. In fact, functoriality is not an {\em a priori} consideration in the proofs
of any of the results cited above.
\end{itemize}

It would be desirable to formulate a theorem that addresses the two problems presented above. 
A good starting point is to work with an $F$ that has greater regularity than in 
Result~\ref{R:bharali}. With this assumption, ideas that are quite different from the one 
(sketchily) outlined in (I) become usable. Before we can state our first result, we need:

\begin{definition}
A subset $S\subset\cplx$ is said to be {\em finitely connected} if 
$\widehat{\cplx}\setminus S$ has finitely many connected components ($\widehat{\cplx}$ denotes
the one-point compactification of $\cplx$).
\end{definition}
\smallskip

Some notation: given a compact set $K\Subset\cplx$, we define the class $\hol(K)$ as
\begin{align}
\hol(K) \ := \ &\text{the set of functions that are holomorphic on (not necessarily fixed)} 
\notag\\
&\text{open neighbourhoods of $K$.} \notag
\end{align}
The sub-class $\hol_\zt(\Dsc)$ in the following result is the set of functions in
$\hol(\Dsc)$ that vanish at $\zt$ (where $\zt\in\Dsc$).

\begin{theorem}\label{T:HolderApprx}
Let $F$ be a complex-valued function of H{\"o}lder class $\smoo^\alpha(\Dsc)$,
$0<\alpha<1$, where $\Dsc\Subset\cplx$ is a closed disc with centre at $0$. Assume that 
we can find a nowhere-dense subset $E\subset\Dsc$, a nowhere-vanishing function $A\in\hol(\Dsc)$, 
positive constants $M,K>0$, and a number $\nu\geq 1$ such that
for each $\zt\in\Dsc\setminus E$, there exist:
\begin{itemize}
\item a function $G_\zt\in\hol_\zt(\Dsc)$ satisfying $|z-\zt|^{-1}|G_\zt(z)|\leq M \ \forall
z\in\Dsc$, and
\item a constant $C_\zt\in S^1$,
\end{itemize}
so that
\begin{equation}\label{E:wedge}
C_\zt(z-\zt)(A\doot F(z)-A\doot F(\zt)+G_\zt(z))\in
\left\{u+iv\in\cplx:u\geq 0, \ |v|\leq Ku^{1/\nu}\right\} \;\; \forall z\in\Dsc.
\end{equation}
Then, $\gra{\Dsc}(F)$ is polynomially convex. If, additionally, $F\in\smoo^{1}(\Dsc)$
and the set $\{z\in\Dsc:\partial_{\zbar} f(z)=0\}$ is nowhere dense and finitely
connected, then $[z,F]_{\Dsc}=\smoo(\Dsc)$.
\end{theorem}
\smallskip

\noindent{Lest the profusion of functions make 
Theorem~\ref{T:HolderApprx} seem very technical, we present the following special case of
Theorem~\ref{T:HolderApprx} that has a more concise statement. Concerning the more general
statement: the reader is directed to the remark
that follows our next result.}

\begin{proposition}\label{P:simple}
Let $F$ be a complex-valued function of H{\"o}lder class $\smoo^\alpha(\Dsc)$, 
$0<\alpha<1$, where $\Dsc\Subset\cplx$ is a closed disc with centre at $0$. Assume that
we can find a nowhere-dense subset $E\subset\Dsc$ and constants $K>0$, $\nu\geq 1$ such that
for each $\zt\in\Dsc\setminus E$, $\exists C_\zt\in S^1$ so that
\begin{equation}\label{E:wedgeSimple}
C_\zt(z-\zt)(F(z)-F(\zt))\in
\left\{u+iv\in\cplx:u\geq 0, \ |v|\leq Ku^{1/\nu}\right\} \;\; \forall z\in\Dsc.
\end{equation}
Then, $\gra{\Dsc}(F)$ is polynomially convex. If, additionally, $F\in\smoo^{1}(\Dsc)$ and 
the set $\{z\in\Dsc:\partial_{\zbar} f(z)=0\}$ is nowhere dense and finitely
connected, then $[z,F]_{\Dsc}=\smoo(\Dsc)$.
\end{proposition}
\smallskip

\begin{remark}\label{Rem:fnctrl}
The function $A\in\hol(\Dsc)$ and the
functions $G_\zt\in\hol_\zt(\Dsc), \ \zt\in\Dsc\setminus E$, make condition
\eqref{E:wedge} more permissive than \eqref{E:wedgeSimple}. In fact, the hypothesis of
Theorem~\ref{T:HolderApprx} is permissive enough to
allow us to recover the well-known fact that if $F\in\hol(\Dsc)$, 
then $\gra{\Dsc}(F)$ is polynomially convex. This will {\em not} follow from the more restrictive 
\eqref{E:wedgeSimple}, or from 
any of the sufficient conditions provided by the results cited above. (Of course, the aim of
those earlier results was to establish that $[z,F]_{\Dsc}=\smoo(\Dsc)$, 
with polynomial convexity being a by-product.) While these are the intuitions that led 
to condition \eqref{E:wedge}, we also find that this condition is functorial --- in 
the sense of the discussion in (II) above --- with respect to any 
$\psi\in{\rm Aut}(\CC)$ that preserves $\Dsc\times\CC$. This is
demonstrated in the final section of this paper. 
\end{remark}

\begin{remark}\label{Rem:O'Farrell}
The interested reader is referred to \cite{o'farrellPreskenis:ap2cv79-80}, in which
a theorem on the polynomial convexity of $\gra{\Dsc}(F)$, for $F\in {\rm Lip}(1,\Dsc)$,
is proved. The ideas central to \cite{o'farrellPreskenis:ap2cv79-80} are different from
those associated with the results cited above. Of interest, however, is a nice survey
in \cite[Section~3]{o'farrellPreskenis:ap2cv79-80} of some
known sufficient conditions for polynomial convexity.
\end{remark}

The primary tool for proving Theorem~\ref{T:HolderApprx} is the following proposition.

\begin{proposition}\label{P:tool}
Let $F\in\smoo(\Dsc)$, where $\Dsc\Subset\cplx$ is a closed disc with centre at $0$,
and let $\zt\in\Dsc$. Suppose there exist a constant $p\geq 2$, a nowhere vanishing function 
$A\in\hol(\Dsc)$, and functions $G,H\in\hol_\zt(\Dsc)$ such that
\begin{equation}\label{E:pinch}
|A\doot F(z)-A\doot F(\zt)+G(z)|^p \ \leq \ \er\left[H(z)(A\doot F(z)-A\doot F(\zt)+G(z))\right] \;\; 
\forall 
z\in\Dsc.
\end{equation}
Then
\[
\widehat{\gra{\Dsc}(F)} \ \subset \ 
\{(z,w)\in\Dsc\times\cplx:|A(z)w-A\doot F(\zt)+G(z)|\leq |H(z)|^{1/(p-1)}\}.
\]
\end{proposition}
\smallskip

\noindent{A note on our notation: given a compact $K\Subset\cplx^n$, $\widehat{K}$ denotes the 
polynomially convex hull of $K$. The idea behind Proposition~\ref{P:tool} is taken from
a step in the proof of Theorem~2.1 in Dieu~\&~Chi's paper \cite{dieuChi:lpccgCC09}. However,
their idea is not quite in a form that we can directly use. Hence, we provide a complete
proof of Proposition~\ref{P:tool} in the next section.}
\smallskip

It also turns out that, {\em when
limited to deducing polynomial convexity}, Theorem~\ref{T:HolderApprx} subsumes Wermer's theorem
in \cite{wermer:ad64}. This is discussed in Section~\ref{S:discuss}. In order motivate our next 
result, it will be helpful
to state the result of Wermer that we have referred to several times already.

\begin{result}[Wermer, \cite{wermer:ad64}]\label{R:Wermer}
Let $\Dsc\Subset\cplx$ be a closed disc. If $F(z)=\zbar+R(z)$ on $\Dsc$ and $R$ satisfies
\[
|R(z)-R(\zt)| \ < \ |z-\zt|
\]
for all $\zt, z$ in $\Dsc$ with $\zt\neq z$, then $[z,F]_{\Dsc}=\smoo(\Dsc)$. In particular,
$\gra{\Dsc}(F)$ is polynomially convex. 
\end{result}
\smallskip

\noindent{One could ask whether we can still infer polynomial convexity if we replace
$\zbar$ by $\zbar^m$, where $m\in\zahl_+$, in Result~\ref{R:Wermer} (provided $0\in\Dsc$). The 
answer to this question (with one essential amendment)  --- as Theorem~\ref{T:WermerII} will 
show --- is, ``Yes.'' Note that if $0\in\Dsc$, $m\geq 2$, and $G$ is differentiable
at $0$, then the origin is a point of complex tangency of $\gra{\Dsc}(F)$; this is what
makes our question a non-trivial one. Given a smooth real surface $\srf\subset\CC$ and a point 
$p\in\srf$ at which $T_p(\srf)$ is a complex line, deciding whether or not $\srf$ is locally
polynomially convex at $p$ is a subtle problem. We will abbreviate the phrase 
``point of complex tangency'' to {\em CR~singularity}. When $\srf$ has an isolated 
CR~singularity at $p\in\srf$ and the order of  contact of $T_p(\srf)$ with $\srf$ at $p$ equals
$2$, 
we now have a nearly
complete understanding of the local polynomial hull of $\srf$ at $p$. This knowledge stems
from the works of Bishop \cite{bishop:dmcEs65}, Forstneri{\v{c}}-Stout \cite{forstnericStout:ncpcs91},
and J{\"o}ricke \cite{joricke:lphdnipp97}. Much less is known when the order of contact of 
$T_p(\srf)$ with $\srf$ at an isolated CR~singularity $p$ is {\em greater than} $2$. Note that
when $p$ is a CR~singularity, there is a complex-affine change of coordinate centered at $p$ 
with respect to which $(\srf,p)$ is locally a graph. It would be the graph of a function of the
form
\begin{equation}\label{E:gr}
F(z) \ = \ \sum_{j=0}^mC_jz^{m-j}\zbar^j + \rem(z),
\end{equation}
where $\rem(z)=O(|z|^{m+1})$, if the order of contact of 
$T_p(\srf)$ with $\srf$ at $p$ is $m$, $m\in\zahl_+$. In \cite[Theorem~1.1]{bharali:sdCRslpc05},
some sufficient conditions were given for $(\srf,p)$ to be locally polynomially convex at $p$,
provided $C_m\neq 0$ in the associated $F$ given by \eqref{E:gr}. These are, however, very
technical conditions, and it would be of interest to see whether alternative conditions could
be obtained with considerably less technical exertion. Such conditions can be derived from the
following}

\begin{theorem}\label{T:WermerII}
Let $\Dsc\Subset\cplx$ be a closed disc with centre at $0$.
If $F(z)=\zbar^m+R(z), \ z\in\Dsc$, where
$m\in\zahl_+$, and $R$ satisfies
\begin{equation}\label{E:WermerII}
|R(z)-R(\zt)| \ < \ |z^m-\zt^m| \;\; \text{for all $z,\zt\in\Dsc : z^m\neq\zt^m$}, 
\end{equation}
then $\gra{\Dsc}(F)$ is polynomially convex.

Additionally, we can conclude that $[z,F]_{\Dsc}=\smoo(\Dsc)$ in the following cases:
\begin{itemize}
\item whenever $m=1$ (with no conditions beyond \eqref{E:WermerII} on $R$);
\item if $m\geq 2$, $R\in\smoo^{1}(\Dsc)$, and $\exists\alpha\in (0,1)$ such that $R$ 
satisfies the stronger estimate:
\[
|R(z)-R(\zt)| \ < \ \alpha|z^m-\zt^m| \;\; \text{for all $z,\zt\in\Dsc : z^m\neq\zt^m$}.
\]
\end{itemize}
\end{theorem}
\smallskip

Note the various similarities of the above theorem with Result~\ref{R:Wermer}. However, 
there does not seem to be any obvious way in which the technique pioneered by Wermer 
in \cite{wermer:ad64} --- i.e. the ideas outlined in (I) above --- can be made to work
when $m\geq 2$ in Theorem~\ref{T:WermerII}. It is Proposition~\ref{P:tool} that provides 
the key ingredient in its proof. This proof will be presented in Section~\ref{S:WermerII}.
We should also mention here De Paepe's generalisation \cite{dePaepe:ad86} of
Minsker's Theorem \cite{minsker:saS-Wtpra76}. While the concerns of 
Theorem~\ref{T:WermerII} are quite different --- e.g., the main result in \cite{dePaepe:ad86}
is a local result --- both results involve dealing with functions of a certain pattern
that vanish to higher order at $0\in\cplx$.  
\medskip

\section{The key proposition}\label{S:tool}

Proposition~\ref{P:tool} is the key proposition on which this paper depends. We devote this
section to its proof. That proof, in turn relies on two propositions --- one by H{\"o}rmander 
and the other by Catlin --- that have been known for a long time, and on a recent 
result of Poletsky. We begin by stating this result.

\begin{result}[Poletsky, \cite{poletsky:Jmam04}]\label{R:poletsky}
Let $K$ be a compact subset in an open set $V\subseteq {\cplx}^n$ and assume that there is a continuous
plurisubharmonic function $u$ on $V$ such that $u=0$ on $K$ and positive on $V \setminus K$. If $v$
is a plurisubharmonic function defined on a neighbourhood $W \subset V$ of $K$ and bounded below on $K$,
then there exists a plurisubharmonic function $v'$ on $V$ that coincides with $v$ on $K$.
\end{result}
\smallskip
  
With this, we are in a position to provide the:
\smallskip

\begin{proof}[Proof of Proposition~\ref{P:tool}] Let $\Delta$ be an open disc in $\cplx$
such that $\Dsc\Subset\Delta$ and $A, G, H\in\hol(\Delta)$. Define 
$\psi:\Delta\times\cplx\lrarw\rea$ by
\begin{multline}
\psi(z,w) \ := \ |A(z)w-A(\zt)F(\zt)+G(z)|^p-\er\left[H(z)(A(z)w-A(\zt)F(\zt)+G(z))\right] \\
\forall (z,w)\in\Delta\times\cplx, \notag
\end{multline}
where $\zt$ and $p$ are as given in the statement of Proposition~\ref{P:tool}. 
Clearly, $\psi\in{\sf psh}(\Delta\times\cplx)$. Let us set $\Gamma:=\gra{\Dsc}(F)$. Then,
$\widehat{\Gamma}$ is polynomially convex. Hence, it follows from a well-known construction
by Catlin \cite{catlin:bbhfwpd78} (see \cite[Proposition~1.3]{sibony:sawpd91} also)
that $\exists u\in{\sf psh}(\CC)\cap\smoo(\CC)$ such that $u=0$ on 
$\widehat{\Gamma}$ and $u>0$ on $\CC\setminus\widehat{\Gamma}$.
\smallskip 

Now note that, as $\widehat{\Gamma}\Subset\Dsc\times\cplx$, $\psi$ is defined on a neighbourhood
of $\widehat{\Gamma}$ and is bounded below on $\widehat{\Gamma}$. Thus, all the conditions in
the hypothesis of Result~\ref{R:poletsky} are satisfied (taking $K$ to be $\widehat{\Gamma}$ and 
$v$ to be $\psi$). Thus, $\exists\Psi\in{\sf psh}(\CC)\cap\smoo(\CC)$ such that
\begin{equation}\label{E:rstrctn}
\left.\Psi\right|_{\widehat{\Gamma}} \ = \ \left.\psi\right|_{\widehat{\Gamma}}.
\end{equation}
We now invoke a result of H\"{o}rmander~\cite[Theorem~4.3.4]{hormander:icasv90} that 
provides an alternative characterisation of polynomial convexity, owing to which:
\[
\widehat{\Gamma} \ = \ \left\{(z,w) \in \CC : U(z,w) \leq  \sup\nolimits_{\Gamma}U \;\;
\forall U\in {\sf psh}(\CC) \right\}. 
\]
From this fact and \eqref{E:rstrctn}, it follows that
\[
\Psi(z,w) \ \leq \ \sup\nolimits_{x\in\Gamma}\Psi(x) \ = \ \sup_{x\in\Gamma}\psi(x) \ \leq 0
\;\; \forall (z,w) \in\widehat{\Gamma}.
\]
The last inequality is a consequence of the fact that, by the inequality \eqref{E:pinch},
$\psi(z,F(z))\leq 0 \ \forall z\in\Dsc$. In other words, owing to \eqref{E:rstrctn}
\begin{align}
\ |A(z)w-A(\zt)F(\zt)+G(z)|^p \ &\leq \ \er\left[H(z)(A(z)w-A(\zt)F(\zt)+G(z))\right] \notag \\
\Longrightarrow \ |A(z)w-A(\zt)F(\zt)+G(z)|^p \ &\leq \ 
|H(z)||A(z)w-A(\zt)F(\zt)+G(z)| \;\; \forall (z,w)\in\widehat{\Gamma}. \notag
\end{align}
Therefore, we conclude
\[
\widehat{\Gamma} \ \subset \
\{(z,w)\in\Dsc\times\cplx:|A(z)w-A(\zt)F(\zt)+G(z)|\leq |H(z)|^{1/(p-1)}\}.
\]
\end{proof}
\medskip

\section{The proof of Theorem~\ref{T:HolderApprx}}\label{S:HolderApprx}

We will need the version of Mergelyan's theorem given below in order to prove the second part
of Theorem~\ref{T:HolderApprx}. This version follows easily from Mergelyan's 
Approximation Theorem \cite{mergelyan:uafcv52} --- see, for instance, Andersson's
observation \cite[Remark~6.4]{andersson:tca97}.
  
\begin{result}[Mergelyan]\label{R:mergelyan}
Suppose $K \subset \cplx$ is compact  and $\widehat{\cplx}\setminus K$ has a finite 
number (say $N$) of components. Choose one point, say $a_j$, from each component. Then,
any $f \in \smoo(K)\cap \hol(int(K))$ can be uniformly approximated
by rational functions with poles only at the points $a_j, \ j=1,\dots,N$.
\end{result}
\smallskip

\noindent{Here $\widehat{\cplx}$ denotes the one-point compactification of $\cplx$.}
\smallskip

\begin{proof}[Proof of Theorem~\ref{T:HolderApprx}] Consider a point $\zt\in\Dsc\setminus E$ and 
let $\Delta_\zt$ be an open disc in $\cplx$
such that $\Dsc\Subset\Delta_\zt$ and $A, G_\zt\in\hol(\Delta)$. Set
\begin{align}
H_\zt(z) \ &:= \ C_\zt(z-\zt), \notag \\
\Lam_\zt(z) \ &:= \ H_\zt(z)(A\doot F(z)-A\doot F(\zt)+G_\zt(z)), \; z\in\Delta_\zt. \notag
\end{align}
Owing to the estimate on $G_\zt$, $\exists\lambda>1$ such that
\[
\sup_{z\in\Dsc}|(z-\zt)(A\doot F(z)-A\doot F(\zt)+G_\zt(z))| \ < \ \lambda \;\; \forall\zt\in
\Dsc\setminus E.
\]
Then, writing $p:=\nu(1+(1/\alpha))$, we estimate
\begin{align}
&\frac{|A\doot F(z)-A\doot F(\zt)+G_\zt(z)|^p}{(K^\nu\lambda+1)\er\left[\Lam_\zt(z)/\lambda\right]} \notag 
\\
&\;\;\;
 \leq \ |A\doot F(z)-A\doot F(\zt)+G_\zt(z)|^p \notag \\
&\quad\qquad
 \times\left[\lambda^{-\nu}[\er(\Lam_\zt(z))]^\nu
	+|\mi(\Lam_\zt(z))|^\nu+\left(K^\nu\er(\Lam_\zt(z))-|\mi(\Lam_\zt(z))|^\nu\right)\right]^{-1}
\notag \\
&\;\;\;
 \leq \ \lambda^\nu \frac{|A\doot F(z)-A\doot F(\zt)+G_\zt(z)|^p}{[\er(\Lam_\zt(z))]^\nu
	+|\mi(\Lam_\zt(z))|^\nu} \notag \\
&\;\;\;
 \leq B\lambda^\nu \frac{|A\doot F(z)-A\doot F(\zt)+G_\zt(z)|^{\nu/\alpha}}{|z-\zt|^\nu} 
	\quad \forall z\in\Dsc\setminus\left(A\doot F+G_\zt-A\doot F(\zt)\right)^{-1}\{0\}. \notag
\end{align}
In the above estimate, the expression $[\er(\Lam_\zt(z))]^\nu$ makes sense because,
by \eqref{E:wedge}, $\er(\Lam_\zt(z))\geq 0 \ \forall z\in\Dsc$, and the second inequality
is again a consequence of \eqref{E:wedge}. Here, $B$ denotes a uniform positive constant.
In the following steps {\em $B$ will denote a positive constant that is independent of
$z$ and $\zt$, but whose specific value changes from line to line.}
\smallskip

Now note that as $A$ is defined on an open neighbourhood of $\Dsc$, $A\doot F\in\smoo^\alpha(\Dsc)$.
Combining this fact with the above estimate gives us
\begin{multline}
\frac{|A\doot F(z)-A\doot F(\zt)+G_\zt(z)|^p}{(K^\nu\lambda+1)\er\left[\Lam_\zt(z)/\lambda\right]}
\ \leq \ B\lambda^\nu\left\{\|A\doot F\|^{\nu/\alpha}_{\smoo^\alpha(\Dsc)}+
	\frac{M^{\nu/\alpha}|z-\zt|^{\nu/\alpha}}{|z-\zt|^\nu}\right\} \\
\forall z\in\Dsc\setminus\left(A\doot F+G_\zt-A\doot F(\zt)\right)^{-1}\{0\}, \;\; 
\forall\zt\in \Dsc\setminus E. \notag
\end{multline}
We have just shown that there exists a constant $B\gg 1$, independent of $\zt\in\Dsc\setminus E$,
such that
\begin{multline}
|A\doot F(z)-A\doot F(\zt)+G_\zt(z)|^p \ \leq \ 
B\er\left[H_\zt(z)(A\doot F(z)-A\doot F(\zt)+G_\zt(z))\right] \\
\forall z\in\Dsc, \;\; \forall\zt\in \Dsc\setminus E. \notag
\end{multline}
Then, Proposition~\ref{P:tool} gives us
\begin{multline}\label{E:hullImp}
\widehat{\gra{\Dsc}(F)} \ \subset \
\{(z,w)\in\Dsc\times\cplx:|A(z)w-A(\zt)F(\zt)+G_\zt(z)|\leq (B|z-\zt|)^{1/(p-1)}\} \\
\forall\zt\in \Dsc\setminus E.
\end{multline}

Let us first consider any $\zt\in\Dsc\setminus E$. Since $A(\zt)\neq 0$, \eqref{E:hullImp} already
tells us that
\begin{equation}\label{E:fibre1}
\widehat{\gra{\Dsc}(F)}\cap(\{\zt\}\times\cplx) \ = \ \{(\zt,F(\zt))\}\;\; \forall\zt\in \Dsc\setminus E.
\end{equation}
Next, consider a $\zt_*\in E$. Since $E$ is nowhere dense, there exists a sequence
$\{\zt_n\}_{n\in\nat}\subset\Dsc\setminus E$ such that $\zt_n\lrarw \zt_*$ as $n\to+\infty$. 
Let $(\zt_*,w_*)$ denote a point in $\widehat{\gra{\Dsc}(F)}\cap(\{\zt_*\}\times\cplx)$. Then,
applying \eqref{E:hullImp} with $\zt=\zt_n,\ n\in\nat$, we get
\begin{align}
& |A(\zt_*)||w_*-F(\zt_*)|-|A(\zt_*)-A(\zt_n)||F(\zt_*)|-|A(\zt_n)||F(\zt_*)-F(\zt_n)|
	-|G_{\zt_n}(\zt_*)| \notag \\
&\qquad
 \leq \ |A(\zt_*)w_*-A\doot F(\zt_n)+G_{\zt_n}(\zt_*)| \notag \\
&\qquad
 \leq \ (B|\zt_*-\zt_n|)^{1/(p-1)} \;\; \forall n\in\nat. \notag
\end{align}
If we now set 
\begin{align}
\mu &:= \inf\nolimits_{\Dsc}|A|,	&&N_1:= \sup\nolimits_{\Dsc}|F|, \notag \\
N_2 &:= \left\|\left. A\right|_{\Dsc}\right\|_{\smoo^\alpha(\Dsc)}, 
&&N_3:= \left\|F\right\|_{\smoo^\alpha(\Dsc)}, \notag
\end{align}
then, we infer that
\begin{multline}
|w_*-F(\zt_*)| \ \leq \ \frac{(B|\zt_*-\zt_n|)^{1/(p-1)}+(N_1N_2+N_2N_3)|\zt_*-\zt_n|^\alpha
				+M|\zt_*-\zt_n|}{\mu} \\
\lrarw 0 \; \text{as $n\to+\infty$}. \notag
\end{multline}
Since $\zt_*$ above was arbitrarily picked from $E$, this implies that 
\begin{equation}\label{E:fibre2}
\widehat{\gra{\Dsc}(F)}\cap(\{\zt_*\}\times\cplx) \ = \ 
\{(\zt_*,F(\zt_*))\}\;\; \forall\zt_*\in E.
\end{equation}
Since $\widehat{\gra{\Dsc}(F)}\subset \Dsc\times\cplx$, \eqref{E:fibre1} and
\eqref{E:fibre2} tell us that $\widehat{\gra{\Dsc}(F)}=\gra{\Dsc}(F)$. This establishes
the first part of Theorem~\ref{T:HolderApprx}.
\smallskip

For any compact set $L\Subset\cplx$, let $\mathscr{R}(L)$ denote the class of functions on
$L$ that are uniformly approximable on $L$ by rational functions whose poles lie outside $L$.
Assuming now that $F\in\smoo^1(\Dsc)$, set 
$L_F:= \{z\in\Dsc : \partial_{\zbar}F(z)=0\}$. In view of our assumptions on the topology of 
the set $L_F$, Result~\ref{R:mergelyan} implies that
\begin{equation}\label{E:ratApx}
\mathscr{R}(L_F)= \smoo(L_F).
\end{equation}
At this stage we can invoke a result of Wermer \cite[page~9]{wermer:pcd65}, which says:
\begin{itemize}
\item[{}] {\em Let $F\in\smoo^{1}(\Dsc)$ and assume that $\gra{\Dsc}(F)$ is polynomially convex. Then
$[z,F]_{\Dsc}$ consists exactly of those continuous functions on $\Dsc$ whose restrictions to
$L_F$ belong to $\mathscr{R}(L_F)$.}
\end{itemize}
We have already established that $\gra{\Dsc}(F)$ is polynomially convex. Thus, in view of
\eqref{E:ratApx}, we deduce that $[z,F]_{\Dsc}=\smoo(\Dsc)$.
\end{proof}
\medskip

\section{The proof of Theorem~\ref{T:WermerII}}\label{S:WermerII}

\noindent{We will emulate many of the notations and computations 
used in the proof of Theorem~\ref{T:HolderApprx}. For each $\zt\in\Dsc$, let
us define
\[
\Lam_\zt(z) \ := \ (z^m-\zt^m)(F(z)-F(\zt)), \; z\in\Dsc.
\]
Now observe that
\begin{equation}\label{E:location}
|\Lam_\zt(z) - |z^m-\zt^m|^2| \ 
\begin{cases}
\leq \ |z^m-\zt^m|^2 &\forall z\in\Dsc, \\
< \ |z^m-\zt^m|^2 &\forall z: z^m\neq\zt^m.
\end{cases}
\end{equation}
This tells us that $\Lam_\zt(z)$ lies in the disc $D(|z^m-\zt^m|^2;|z^m-\zt^m|^2)$ whenever 
$z^m\neq\zt^m$. Thus, if $r$ denotes the radius of $\Dsc$, then
\[
\Lam_\zt(z) \in \overline{D(4r^{2m};4r^{2m})} \;\; 
\forall (z,\zt)\in\Dsc\times\Dsc.
\]
It is elementary to infer from this that
\begin{equation}
\Lam_\zt(z) \in \{u+iv\in\cplx:u\geq 0, \ |v|\leq\sqrt{8}r^mu^{1/2}\} \;\;
\forall (z,\zt)\in\Dsc\times\Dsc.
\end{equation}
Note the resemblance of the above to condition to the condition 
\eqref{E:wedge} in Theorem~\ref{T:HolderApprx}. Hence, just as in the proof
of Theorem~\ref{T:HolderApprx}, if we pick a $\lambda>1$ such that $\lambda$
satisfies
\[
\sup_{\Dsc}|\Lam_\zt(z)| \ < \ \lambda \;\; \forall \zt\in\Dsc,
\]
then a string of estimates analogous to the one in the proof of Theorem~\ref{T:HolderApprx}
leads to
\[
\frac{|F(z)-F(\zt)|^4}{(8r^{2m}\lambda+1)\er[\Lam_\zt(z)/\lambda]} 
\ \leq \ \lambda^2\frac{|F(z)-F(\zt)|^2}{|z^m-\zt^m|^2}
\;\; \forall z\in\Dsc: z^m\neq\zt^m.
\]
Applying the condition \eqref{E:WermerII} to the above estimate, we get
\begin{equation}\label{E:compare}
|F(z)-F(\zt)|^4 \ \leq \ 
4\lambda(8r^{2m}\lambda+1)\er\left[(z^m-\zt^m)(F(z)-F(\zt))\right] \;\;
\forall (z,\zt)\in\Dsc\times\Dsc.
\end{equation}}
\smallskip

Let us denote the uniform constant on the right-hand side of \eqref{E:compare} by $B$.
Combining this inequality with Proposition~\ref{P:tool}, we get:
\begin{equation}\label{E:hullPinch}
\widehat{\gra{\Dsc}(F)} \ \subset \
\{(z,w)\in\Dsc\times\cplx:|w-F(\zt)|\leq (B|z^m-\zt^m|)^{1/3}\} \;\; \forall\zt\in\Dsc.
\end{equation}
Clearly, 
$\widehat{\gra{\Dsc}(F)}\cap(\{\zt\}\times\cplx)=\{(\zt,F(\zt))\} \;\; \forall\zt\in\Dsc$.
We have thus established that $\gra{\Dsc}(F)$ is polynomially convex. 
\smallskip

The second assertion of Theorem~\ref{T:WermerII} for the case $m=1$ is just the conclusion
of Result~\ref{R:Wermer}. Let us thus consider the case when $m\geq 2$. Let us fix a $z\neq 0$.
Then, we can find a sequence $\{r_n\}_{n\in\nat}\subset(\mathbb{R}\setminus\{0\})$ such that
\begin{itemize}
\item $\lim_{n\to\infty}r_n=0$; 
\item $(z+r_n)^m\neq z^m \ \forall n\in\nat$;
\item $(z+ir_n)^m\neq z^m \ \forall n\in\nat$.
\end{itemize}
By our assumptions on $R$:
\begin{equation}\label{E:barDeriv}
\partl{F}{\zbar}(z) \ = \ m\zbar^{m-1}+
\lim_{n\to\infty}\frac{1}{2}\left(\frac{R(z+r_n)-R(z)}{r_n}-\frac{R(z+ir_n)-R(z)}{ir_n}\right).
\end{equation}
Note that, by the properties of the sequence $\{r_n\}_{n\in\nat}$, we can find an
$N\in\zahl_+$ sufficiently large that:
\begin{align}
\frac{|R(z+r_n)-R(z)|}{|r_n|} \ < \ \alpha\frac{|(z+r_n)^m-z^m|}{|r_n|}
\ &\leq \ \alpha\sum_{j=0}^{m-1}|z+r_n|^j|z|^{m-1-j} \notag \\
&\leq \ m\frac{1+\alpha}{2}|z|^{m-1} \;\; \forall n\geq N. \label{E:estX}
\end{align}
A similar estimate holds for the second difference quotient in \eqref{E:barDeriv}. Combining
\eqref{E:estX} with \eqref{E:barDeriv}, we get
\[
\left|\partl{F}{\zbar}(z)\right| \ \geq \ m|z|^{m-1}-m\frac{1+\alpha}{2}|z|^{m-1} \ > \ 0,
\]
provided $z\neq 0$. Hence, in the terminology of the previous section,
\[
L_F \ := \{z\in\Dsc : \partial_{\zbar}F(z)=0\} \ = \ \{0\}.
\]
We conclude with the same argument as in the final paragraph of the proof of 
Theorem~\ref{T:HolderApprx}. Since we have already established that $\gra{\Dsc}(F)$ is 
polynomially convex, and $L_F$ is a singleton, \cite{wermer:pcd65} implies that
$[z,F]_{\Dsc}=\smoo(\Dsc)$. \hfill \qed
\medskip

\section{The relation of Theorem~\ref{T:HolderApprx} to known results}\label{S:discuss}

This section is dedicated to elaborating upon two observations made in Section~\ref{S:intro}.

\subsection{The invariance of the hypothesis of Theorem~\ref{T:HolderApprx} under the
action of certain elements of $\boldsymbol{{\rm Aut}(\CC)}$:} Let $\Dsc$ and $F$ be as given in 
Theorem~\ref{T:HolderApprx}. It is a tautology that polynomial convexity of $\gra{\Dsc}(F)$
(or the lack thereof) is preserved when transformed by any $\psi\in{\rm Aut}(\CC)$ that
maps $\Dsc\times\cplx$ onto itself. We claim that the truth (or falsity) 
of the hypothesis of Theorem~\ref{T:HolderApprx} too is preserved when $F$ is replaced by
that function of which $\psi(\gra{\Dsc}(F))$ is a graph. To establish this, 
it suffices to accomplish the following:
\begin{itemize}
\item given a $\psi\in{\rm Aut}(\CC)$ that maps $\Dsc\times\cplx$ onto itself,
show that there is a function (in the notation of (II), Section~\ref{S:intro}) 
$\widetilde{{}^\psi F}\in\smoo^\alpha(\Dsc)$ such that 
$\psi(\gra{\Dsc}(F))=\gra{\Dsc}(\widetilde{{}^\psi F})$.
\smallskip

\item assuming that $F$ satisfies the hypothesis of Theorem~\ref{T:HolderApprx}, produce
a nowhere dense $\widetilde{E}\subset\Dsc$; a constant $\widetilde{M}>0$;
a function $\widetilde{A}\in\hol(\Dsc)$ that vanishes nowhere on $\Dsc$; and, associated
to each $\zt\in\Dsc\setminus\widetilde{E}$, constants $\widetilde{C_\zt}\in S^1$ and 
functions $\widetilde{G_\zt}\in\hol_\zt(\Dsc)$ such that the statement obtained
by replacing $F$ by $\widetilde{{}^\psi F}$, and the other objects occurring in the 
hypothesis of Theorem~\ref{T:HolderApprx} by their analogues listed above, is also true.
\end{itemize}
  
Let us set
\[
\mathfrak{G}(\Dsc) \ := \ \{\psi\in{\rm Aut}(\CC):\psi(\Dsc\times\cplx)=\Dsc\times\cplx\}
\]
and determine all the elements of $\mathfrak{G}(\Dsc)$. For a $\psi\in\mathfrak{G}(\Dsc)$,
write $\psi=(\psi_1,\psi_2)$. For each $z\in\Dsc$, the functions $\psi_1(z,\bcdot)$ are
entire functions that map $\cplx$ to $\Dsc\Subset\cplx$. Hence, it follows from Liouville's
theorem that $\psi_1$ depends only on $z$. This implies that $\psi_1\in{\rm Aut}(\cplx)$ and,
as $\psi_1(\Dsc)=\Dsc$, $\exists\phi\in\mathbb{R}$ such that
\[
\psi_1(z,w) \ = \ \cis{\phi}z \;\; \forall (z,w)\in\CC.
\]
Furthermore, since $\psi$ must be injective, it follows
that $\psi_2(z,\bcdot)\in{\rm Aut}(\cplx)$ for each $z\in\CC$. It can now easily
be shown that there exist entire functions $\gee, \hft$, 
where $\gee$ is nowhere-vanishing, such that
\begin{equation}\label{E:auts}
\psi\in\mathfrak{G}(\Dsc) \ \Longrightarrow \ 
\psi(z,w)=(\cis{\phi}z,\gee(z)w+\hft(z)) \;\; \forall (z,w)\in\CC.
\end{equation}
It is, of course, obvious that each $\psi$ having the form given in \eqref{E:auts}
belongs to $\mathfrak{G}(\Dsc)$.
\smallskip

Now assume that $F$ satisfies the hypothesis on Theorem~\ref{T:HolderApprx}.
In view of \eqref{E:auts}, if $\psi\in\mathfrak{G}(\Dsc)$, then
\[
\widetilde{{}^\psi F}(z) \ = \gee(\ncis{\phi}z)F(\ncis{\phi}z)+\hft(\ncis{\phi}z)
\;\; \forall z\in\Dsc,
\]
for some $\phi\in\mathbb{R}$ and $\gee,\hft\in\hol(\cplx)$ such that 
$\gee$ is non-vanishing. Since $\gee,\hft\in\hol(\cplx)$, it is immediate that
$\widetilde{{}^\psi F}\in\smoo^\alpha(\Dsc)$. 
\smallskip

Let us introduce two notations. Let
\begin{align}
\widetilde{E} &:= \cis{\phi}E, \notag \\
\mathcal{W}(K,\nu) &:= \left\{u+iv\in\cplx:u\geq 0, \ |v|\leq Ku^{1/\nu}\right\}. \notag
\end{align}
Then, for each $\zt\in\Dsc\setminus\widetilde{E}$, write
\[
\Lam_{\zt}(z) \ := \ 
C_{\ncis{\phi}\zt}\ncis{\phi}(z-\zt)(A\doot F(\ncis{\phi}z)-A\doot F(\ncis{\phi}\zt)+
G_{\ncis{\phi}\zt}(\ncis{\phi}z)),
\]
where $C_{\ncis{\phi}\zt}\in S^1$ and $G_{\ncis{\phi}\zt}\in\hol(\Dsc)$ are {\em exactly}
as provided by the hypothesis of Theorem~\ref{T:HolderApprx}; this makes sense because 
$\zt\in\Dsc\setminus\widetilde{E}$ in the preceding definition. It is obvious, given our 
assumption that $F$ satisfies the hypothesis of Theorem~\ref{T:HolderApprx}, that
\begin{equation}\label{E:lamWedge}
\Lam_{\zt}(z) \in \mathcal{W}(K,\nu) \;\; \forall z\in\Dsc \ 
\text{and $\forall\zt\in\Dsc\setminus\widetilde{E}$.}
\end{equation}
It is now follows from a simple computation that if we define
\begin{align}
\widetilde{A} &:= \frac{A}{\mathcal{G}}(\ncis{\phi}\bcdot), 
&&\widetilde{G_\zt} :=  G_{\ncis{\phi}\zt}(\ncis{\phi}\bcdot)
				-\frac{A\doot\mathcal{H}}{\mathcal{G}}(\ncis{\phi}\bcdot)
				+\frac{A\doot\mathcal{H}}{\mathcal{G}}(\ncis{\phi}\zt), \notag \\
\widetilde{C_\zt} &:= \ncis{\phi}C_{\ncis{\phi}\zt} 
&&\forall\zt\in\Dsc\setminus\widetilde{E} \notag,
\end{align}
then 
\begin{equation}\label{E:lamTrans}
\Lam_{\zt}(z) \ = \ 
\widetilde{C_\zt}(z-\zt)
	(\widetilde{A}\doot{\widetilde{{}^\psi F}}(z)-\widetilde{A}\doot{\widetilde{{}^\psi F}}(\zt)
	+ \widetilde{G_\zt}(z)) \;\; \forall z\in\Dsc \
	\text{and $\forall\zt\in\Dsc\setminus\widetilde{E}$}.
\end{equation}
From \eqref{E:lamWedge} and \eqref{E:lamTrans}, it follows that if $F$ satisfies the hypothesis
of Theorem~\ref{T:HolderApprx}, then the following statement: 
\begin{itemize}
\item[{}] {\em We can find a nowhere-dense subset 
$\widetilde{E}\subset\Dsc$, a nowhere vanishing function 
$\widetilde{A}\in\hol(\Dsc)$, positive constants $\widetilde{M},K>0$, and a number $\nu\geq 1$
such that for each $\zt\in\Dsc\setminus\widetilde{E}$, there exist:
\begin{itemize}
\item[\textbullet]
 a function $\widetilde{G_\zt}\in\hol_\zt(\Dsc)$ satisfying 
$|z-\zt|^{-1}|\widetilde{G_\zt}(z)|\leq \widetilde{M} \ \forall z\in\Dsc$, and
\item[\textbullet] 
a constant $\widetilde{C_\zt}\in S^1$,
\end{itemize} 
so that
\[
\widetilde{C_\zt}(z-\zt)
(\widetilde{A}\doot{\widetilde{{}^\psi F}}(z)-\widetilde{A}\doot{\widetilde{{}^\psi F}}(\zt)
+ \widetilde{G_\zt}(z))
\in
\mathcal{W}(K,\nu) \;\; \forall z\in\Dsc.
\]}
\end{itemize}
also holds true (with $K$ and $\nu$ having exactly the same values as hypothesized in 
Theorem~\ref{T:HolderApprx}). As $\psi$ is invertible, the last implication is sufficient
to establish that the hypothesis of Theorem~\ref{T:HolderApprx} is functorial with
respect to $\mathfrak{G}(\Dsc)$.
\medskip

\subsection{The relation between Theorem~\ref{T:HolderApprx} and Wermer's theorem 
(Result~\ref{R:Wermer}):} We remark that when we consider the weaker form of Wermer's
theorem --- i.e. taking its conclusion as {\em merely} that $\gra{\Dsc}(F)$ is
polynomially convex --- then this weaker result is subsumed by Theorem~\ref{T:HolderApprx}.
To see this, we first observe that, without loss of generality, we may take 
$\Dsc$ in Wermer's theorem to be a disc centred at $0\in\cplx$. Next, we draw the 
reader's attention to the argument in the first paragraph of
the proof of Theorem~\ref{T:WermerII}. By this argument, and from the assumptions on
$F$ in Wermer's theorem, it follows that 
\[
(z-\zt)(F(z)-F(\zt))\in 
\left\{u+iv\in\cplx:u\geq 0, \ |v|\leq \sqrt{8}ru^{1/2}\right\}
\]
where $r>0$ is the radius of the disc $\Dsc$. The above shows that the functions
considered by Result~\ref{R:Wermer} are a special case of the functions considered
by Theorem~\ref{T:HolderApprx}.
\medskip

\noindent{{\bf Acknowledgement.} I thank the anonymous referee of an earlier version of
this work for helpful comments on exposition and for drawing my attention to work of
which I had been unaware.}

\end{document}